\documentclass[draft,reqno]{amsart}

\usepackage{amsmath}
\usepackage{amssymb}
\usepackage[all]{xy}
\usepackage{picinpar}
\usepackage{palatino}
\usepackage[usenames,dvipsnames,svgnames,table]{xcolor}
\usepackage{mathtools}

\def\eu{\mathfrak}
\def\ma{\mathbb}
\def\mc{\mathcal}

\def\p{P_{\infty}}
\def\f{{\ma F}_q^{\ast}}
\def\F{{\ma F}_q}

\def\fin{\hfill\qed\bigskip}
\def\ge#1{#1_{\eu{gex}}}
\def\gv#1{#1_{{\eu{ge}},w}}
\def\gev#1{#1_{{\eu{gex}},w}}
\def\g#1{#1_{\eu{ge}}}
\def\*#1{#1^*}
\def\cicl#1#2{k(\Lambda_{{#1}^{#2}})}

\def\polyt#1{P_1^{\alpha_{1,#1}}\cdots P_r^{\alpha_{r,#1}}}
\def\polynr#1{{\mc P}_{#1,1} \cdots {\mc P}_{#1,s_#1}}
\def\polyn#1{{\mc P}_1^{\alpha_1}{\mc P}_2^{\alpha_2}\cdots {\mc P}_{#1-1}^{\alpha_{#1-1}}
{\mc P}_{#1}^{\alpha_{#1}}}
\def\polynn#1{{\mc P}_2^{\alpha_1}{\mc P}_3^{\alpha_2}\cdots {\mc P}_{#1}^{\alpha_{#1-1}}
{\mc P}_{1}^{\alpha_{#1}}}
\def\polynnn#1{{\mc P}_1^{{\eta}^{#1-1}}{\mc P}_2^{{\eta}^{#1-2}}\cdots {\mc P}_{#1-1}^{\eta}
{\mc P}_{#1}}
\def\pol#1#2{{\mc P}_{#1,1}^{\alpha_{#1,1,#2}}{\mc P}_{#1,2}^{\alpha_{#1,2,#2}}
\cdots {\mc P}_{#1,s_{#1-1}}^{\alpha_{#1,s_#1-1,#2}}{\mc P}_{#1,s_{#1}}^{\alpha_{#1,s_#1,#2}}}
\def\poly#1{{\mc P}_{#1,1}^{\eta^{s_#1-1}}{\mc P}_{#1,2}^{\eta^{s_#1-2}}
\cdots {\mc P}_{#1,s_{#1-1}}^{\eta}{\mc P}_{#1,s_{#1}}}

\def\v#1{#1_w}
\def\Ku#1#2{k\big(\sqrt[#1]{#2}\big)}
\def\Kuu#1#2{k_w\big(\sqrt[#1]{#2}\big)}

\def\z#1{\zeta_{l^#1}}

\def\Witt#1{\stackrel{_{\bullet}}{#1}}

\newcommand{\ord}{\operatorname{ord}}
\newcommand{\Gal}{\operatorname{Gal}}
\newcommand{\lcm}{\operatornamewithlimits{lcm}}

\newcommand{\con}{\operatorname{con}}

\newcounter{bean}
\newcounter{2bean}

\def\l{
\begin{list}
{\rm{(\alph{bean}).-}}{\usecounter{bean}
\setlength{\labelwidth}{0.8in}
\setlength{\labelsep}{0.3cm}
\setlength{\leftmargin}{1cm}}}

\def\las{\begin{list}
	{{\rm {(\arabic{2bean})}}}{\usecounter{2bean}
\setlength{\labelwidth}{0.8in}
\setlength{\labelsep}{0.3cm}
\setlength{\leftmargin}{1cm}}}

\numberwithin{equation}{section}
\newtheorem{theorem}{Theorem}[section]
\newtheorem{proposition}[theorem]{Proposition}

\newtheorem{remark}[theorem]{Remark}

\title[Function field genus theory for non-Kummer extensions]
{Function field genus theory for non-Kummer extensions}
 
\author[M. Rzedowski]{Martha Rzedowski--Calder\'on}
\address{Departamento de Control Autom\'atico\\
Centro de Investigaci\'on y de Estudios Avanzados del I.P.N.}
\email{mrzedowski@ctrl.cinvestav.mx}

\author[G. Villa]{Gabriel Villa--Salvador}
\address{Departamento de Control Autom\'atico\\
Centro de Investigaci\'on y de Estudios Avanzados del I.P.N.}
\email{gvillasalvador@gmail.com, gvilla@ctrl.cinvestav.mx}

\subjclass[2010]{Primary 11R58; Secondary 11R29, 11R60}

\keywords{Global fields, genus fields, non-Kummer extensions,
cyclic extensions, abelian extensions}

\date{April 4, 2022}

\begin{document}

\begin{abstract}

In this paper we first obtain the genus field of a finite abelian
non-Kummer
$l$--extension of a global rational function field. Then, using that
the genus field of a composite of two abelian extensions
of a global rational function field
with relatively prime degrees is equal to the composite
of their respective genus fields and our previous results,
we deduce the general
expression of the genus field of a finite abelian extension
of a global rational function field.

\end{abstract}

\maketitle

\section{Introduction}\label{S1}

The concepts of  Hilbert Class Field (HCF) $K_H$ and
narrow or extended Hilbert Class Field 
$K_{H^+}$, of a fixed number field $K$,
are canonically defined as the maximal abelian unramified 
abelian extension of $K$ and as
the maximal abelian extension unramified at the finite primes of $K$,
respectively. The function field case is different because the direct definition
of the HCF over a global function field $K$ as the maximal unramified
abelian extension has the inconvenience of being of infinite degree over
$K$. 

There are several different possible definitions of HCF of a global
function field $K$. The one we will be using is that for a fixed
finite nonempty set $S$ of places of $K$, the HCF
$K_{H}$ of $K$ is defined
as the maximal unramified abelian extension of $K$ where all the
places of $S$ decompose fully. The extension $K_H/K$ is a 
finite extension with Galois group isomorphic, via the Artin
Reciprocity Law, to the class group $Cl_S$ of the Dedekind
domain consisting of the elements of $K$ regular away from $S$.
The genus field of $K$ over a subfield $k$ is defined as $\g K= K\*k$,
where $\*k$ is the maximal abelian extension of $k$ contained in
$K_H$.

We are interested in the case of the rational function
field $k=\F(T)$ and $K/k$ a finite abelian extension. In \cite{Leo53},
H. Leopoldt studied the extended genus field $\ge K$ of a
finite abelian extension of the field of rational numbers
${\ma Q}$, by means of Dirichlet characters. Using Leopoldt's technique
we applied Dirichlet characters to the function field case and
found a general description of $\g K$ (\cite{BaMoReRzVi2018,
MaRzVi2013, MaRzVi2017}). In these papers it was also provided
an explicit description of $\g K$ in the cases of a Kummer cyclic
extension of prime degree $l$ and of an abelian $p$-extension where
$p$ is the characteristic of $k$. In \cite{BaRzVi2013, MoReVi2019, RzeVil2021},
the explicit description of $\g K$ was given when $K/k$ is a finite
Kummer $l$-extension with $l$ a prime number.

In this paper, we study the explicit description of the genus
field $\g K$  of the remaining
case: $K/k$ a finite abelian non-Kummer $l$-extension
with $l\neq p$, the characteristic of $K$. Using
this explicit description and the results of \cite{BaMoReRzVi2018}
and \cite{RzeVil2021},
we have the explicit description of $\g K$ of any finite abelian extension
$K/k$. By explicit description, we mean to give $\gv K$ in terms
of radical extensions, where $\gv K/\g K$ is the extension of constants
$\gv K=\g K{\ma F}_{q^w}$.

The main tool to find $\gv K$ is that, given $K/k$ a finite $l$-cyclic
non-Kummer extension with $l$ a prime other than the characteristic,
if $k_w$ denotes the extension of constants adjoining all the relevant
roots of unity, then we find explicitly ${\mc D}$ in the ring of integers
of $k_w$ such that $K_w=k_w(\sqrt[l^n]{{\mc D}})$.
Then we generalize the technique to a general finite 
abelian $l$-extension.
Our main result is Theorem \ref{T4.2}.

Theorem \ref{T5.2} gives the general description of $\g K$ for a
finite abelian extension $K/k$.
One crucial result that allows us to be able to give explicitly this
general description is that if $K_1/k$ 
and $K_2/k$ are two finite abelian extensions
of relatively prime degrees, we have $\g{(K_1)}\g{(K_2)}=
\g{(K_1K_2)}$.

\section{Antecedents and general notations}\label{S2}

Let $k=\F(T)$ be a global rational function field,
where $\F$ is the finite field of $q$ elements,
$R_T=\F[T]$ denotes the polynomial ring that may be
considered as the ring of integers of $k$. Let $R_T^+$ be the
set of the monic irreducible elements of $R_T$ or, equivalently,
the ``{\em finite}'' primes of $k$. 

For the Carlitz--Hayes theory 
of cyclotomic function fields, we will be using
\cite[Ch. 12]{Vil2006} and \cite[Cap. 9]{RzeVil2017}.
For $N\in R_T$, $\cicl N{}$ denotes
the $N$--th cyclotomic function field where $\Lambda_N$ is
the $N$--th torsion of the Carlitz module. For $D\in R_T$ we
denote $\*D:=(-1)^{\deg D} D$.

The results on genus fields of function fields can all be found
in \cite{BaMoReRzVi2018,MaRzVi2013,MaRzVi2017}
and \cite[Cap. 14]{RzeVil2017}. For the explicit description
of genus fields, we refer to \cite{BaMoReRzVi2018, MoReVi2019,
RzeVil2021}.

We denote the infinite prime of $k$ by $\p$. That is, $\p$ is the pole
divisor of $T$ and $1/T$ is a uniformizer for $\p$. The ramification 
index of $\p$ in $\cicl N{}/k$ is equal to
$q-1$ and the inertia degree of $\p$
in every cyclotomic function field is always equal to $1$.

For any extension $L/K$ with $L/k$ a finite abelian extension,
$e_{\infty}(L/K)$ denotes the ramification index of the infinite
primes of $K$, that is, the primes of $K$ dividing $\p$ and $f_{
\infty}(L/K)$ denotes the inertia degree of the infinite primes.
Similarly $e_P$ and $f_P$ for a finite prime $P$ of $k$.

Let $F$ be any cyclotomic function field, that is, $F\subseteq
\cicl N{}$ for some $N\in R_T$. Then $\g F=M^+ F$, where
$M$ is the maximal cyclotomic extension of $F$ unramified at the finite
primes, and $M^+$ denotes the ``{\em real subfield}''
of $M$, that is, the decompositon field of $\p$.
We denote $\ge F=M$ the extended genus field of $F/k$. Then
$\g F=\ge F^+ F$.
We have $\g F\subseteq \ge F\subseteq \cicl N{}$ and
$\ge F/\g F$ is totally ramified at the
infinite primes. We have that
$[\ge F:\g F]=e_{\infty}(\ge F/\g F)=e_{\infty}(\ge F/ F)$. Therefore, to
obtain $\g F$ we need to compute a suitable subextension of $\ge F$ of
degree $e_{\infty}(\ge F/F)$.

When $K/k$ is a finite abelian extension, it follows from
the Kronecker--Weber Theorem that 
there exist $N\in R_T$, $n'\in{\ma N}\cup\{0\}$ and $m'\in{\ma N}$
such that $K\subseteq {_{n'}\cicl N{}_{m'}}$, where for any
field $F$ containing $k$, $F_{m'}:=F
{\ma F}_{q^{m'}}$ is the extension of
constants, and $_{n'}F:=FL_{n'}$, with $L_{n'}$ 
the maximal subfield of $\cicl {1/T^{n'}}{}$,
where $\p$ is totally and wildly ramified.
Then we define 
\begin{gather}\label{Ec1}
E:=K {\mc M}\cap \cicl N{},
\end{gather}
where ${\mc M}=L_{n'}k_{m'}$. We have that
$\g K=\g E^H K$, where $H$ is the decomposition group of
the infinite primes in $\g E K/K$ and of $EK/K$ 
(see \cite{BaMoReRzVi2018}).

We only consider geometric extensions $K/k$,
that is, $\F$ is the field of constants of $K$.

For $v\in{\ma N}$, $C_v$ will denote the cyclic group of $v$ elements and
$\z n$ will denote a primitive $l^n$-th root of unity in a finite field. We will
use the notation $\con_{F/E}$ for the conorm map from $F$ to $E$.

\section{Extensions of $k$ of prime power degree}\label{S3}

In this section we consider, for a prime number $l$ a finite abelian
$l$-extension $K/k$ of exponent $l^n$, $n\in {\ma N}$.

\subsection{Case $l=p$}\label{S3.1}

In \cite[Corollary 6.6]{BaMoReRzVi2018} we found the
explicit genus field $\g K$ of a finite abelian $p$-extension
$K/k$. Let $\Witt +$,
$\Witt -$ and $\Witt \cdot$ be the Witt operations.
Let $P_1,\ldots,P_r\in R_T^+$ be the finite primes in $k$
ramified in $K$. Given a Witt vector $\vec \xi_0$,
we may decompose $\vec\xi_0$ as
\begin{gather}\label{Eq3.0}
\vec \xi_0={\vec\delta}_1
\Witt + \cdots \Witt + {\vec\delta}_r \Witt + \vec\gamma,
\end{gather}
where $\delta_{i,j}=\frac{Q_{i,j}}{P_i^{e_{i,j}}}$, $e_{i,j}\geq 0$, $Q_{i,j}\in R_T$
and if $e_{i,j}>0$, then $e_{i,j}=
\lambda_{i,j}p^{m_{i,j}}$, $\gcd (\lambda_{i,j},p)=1$,
$0\leq m_{i,j}< n$, $\gcd(Q_{i,j},P_i)=1$ and
$\deg (Q_{i,j})<\deg (P_i^{e_{i,j}})$, and $\gamma_j=f_j(T)\in R_T$ with
$\deg f_j=\nu_j p^{m_j}$ and $\gcd(q,\nu_j)=1$, $0\leq m_j<n$
when $f_j\not\in k_0$.

\begin{theorem}\label{T3.1} Let $K/k$ be a finite 
abelian $p$--extension with Galois group $\Gal(K/k)=G
\cong G_1\times\cdots\times G_s$ with $G_i\cong
C_{p^{\alpha_i}}$, $1\leq i\leq s$. Let $K$
be the composite $K=K_1\cdots K_s$ such that
$\Gal(K_i/k)\cong G_i$. Let $P_1,\ldots, P_r$
be the finite primes ramified in $K/k$.
Let $K_i=k(\vec w_i)$ be given
by the equation
\begin{gather*}
\vec w_i^p\Witt - \vec w_i=\vec \xi_i,\quad 1\leq i\leq s.
\intertext{Write each $\vec\xi_i$ as in {\rm{(\ref{Eq3.0})}} that is,}
\vec \xi_i={\vec\delta}_{i,1}
\Witt + \cdots \Witt + {\vec\delta}_{i,r} \Witt + \vec\gamma_i,
\intertext{such that all the components of $\vec\delta_{i,j}$ are written
so that the degree of the numerator is less than the degree
of the denominator, the support of
the denominator is at most $\{P_j\}$ and the components
of $\vec\gamma_i$ are polynomials. Let}
\vec w_{i,j}^p\Witt -\vec w_{i,j}=\vec\delta_{i,j},\quad
1\leq i\leq s,\quad 1\leq j\leq r
\intertext{and}
\vec z_i^p\Witt -\vec z_i=\vec \gamma_i, \quad 1\leq i\leq s.
\intertext{Then, the genus field $\g K$ of $K$ is given by}
\g K=k\big(\vec w_{i,j},\vec z_i\mid 1\leq i\leq s, 1\leq j\leq r\big).
\tag*{$\fin$}
\end{gather*}
\end{theorem}

\subsection{Case $l\neq p$, $K/k$ Kummer}\label{E3.2}

We now consider the case of a finite abelian Kummer $l$-extension
$K/k$ of exponent $l^n$. That is $l^n|q-1$ or, equivalently,
$\z n\in \F$.

Let $K=k\big(\sqrt[l^{n_1}]{\gamma_1 D_1},
\cdots, \sqrt[l^{n_s}]{\gamma_s D_s}\big)=K_1\cdots K_s$, where
$K_{\varepsilon}=k\big(\sqrt[l^{n_{\varepsilon}}]{\gamma_{
\varepsilon} D_{\varepsilon}}\big)$, $D_{\varepsilon}\in R_T$ monic,
$\gamma_{\varepsilon}\in \*{\F}$, $1\leq \varepsilon\leq s$ and
$n=n_1\geq \ldots\geq n_s$.
Let $P_1,\ldots,P_r$ be the finite primes ramified in $K/k$
and let
\[
D_{\varepsilon}=\polyt {\varepsilon}\quad\text{with}\quad 0\leq
\alpha_{j,\varepsilon}\leq l^{n_{\varepsilon}}-1,\quad 1\leq j\leq r,
\quad 1\leq \varepsilon\leq s,
\]
where $\alpha_{j,\varepsilon}=b_{j,\varepsilon}' l^{a'_{j,\varepsilon}}$
with $\gcd(l,b'_{j,\varepsilon})
=1$ and $\deg P_{j}=c'_{j}l^{d'_{j}}$ with $\gcd(l,c'_{j})=1$, 
$1\leq j\leq r$. 
We have, for $P_j$, that
\begin{gather}
e_{P_j}(K/k)=\lcm_{1\leq\varepsilon\leq s}
[e_{P_j}(K_{\varepsilon}/k)]=l^{\beta_j}\nonumber
\intertext{with}
\beta_j=\max_{1\leq \varepsilon\leq s}\{n_{\varepsilon}-
v_l(\alpha_{j,\varepsilon})\}=\max_{\substack{1\leq \varepsilon\leq s\\
b'_{j,\varepsilon}\neq 0}} \{n_{\varepsilon}-a'_{j,\varepsilon}\}.\label{Ec3.5}
\intertext{Also, we have}
\begin{align}
l^{t'}:&=e_{\infty}(K/k)=\lcm_{1\leq \varepsilon\leq s}
[e_{\infty}(K_{\varepsilon}/k)]=
\lcm_{1\leq \varepsilon\leq s}\Big[\frac{l^{n_{\varepsilon}}}{\gcd(
l^{n_{\varepsilon}},\deg D_{\varepsilon})}\Big]\nonumber\\
&=\lcm_{1\leq \varepsilon\leq s}\big[l^{n_{\varepsilon}-\min\{
n_{\varepsilon},v_l(\deg D_{\varepsilon})\}}\big],\label{Ec3.6}
\end{align}
\end{gather}
that is, $t'=\max\limits_{1\leq\varepsilon\leq s}\big\{n_{\varepsilon}-
\min\{n_{\varepsilon},v_l(\deg D_{\varepsilon})\}\big\}$.

\begin{theorem}[Kummer case]\label{T3.2} 
Let $K/k$ be a finite Kummer $l$--extension 
of $k$. Let
$\Gal(K/k)\cong C_{l^{n_1}}\times \cdots \times C_{l^{n_s}}$
with $n=n_1\geq n_2\geq \cdots \geq n_s$ and $l^n|q-1$.
We have $K=K_1\cdots K_s$,
$K_{\varepsilon}=\Ku{l^{n_{\varepsilon}}}{\gamma_{
\varepsilon}D_{\varepsilon}}$,
$D_{\varepsilon}\in R_T$ monic and $\gamma_{\varepsilon}
\in \f$, $1\leq \varepsilon\leq s$. 
Let $P_1,\ldots,P_r$ be the finite primes in $k$ ramified in
$K$ with $P_1,\ldots,P_r\in R_T^+$ distinct. Let 
\begin{gather*}
e_{P_j}(K/k)=l^{\beta_j},\quad
1\leq \beta_j\leq n, \quad 1\leq j\leq r,\quad\text{and}\quad
e_{\infty}(K/k)=l^{t'}, \quad 0\leq t'\leq n
\end{gather*}
given by {\rm{(\ref{Ec3.5})}} and {\rm{(\ref{Ec3.6})}} and let
$\deg P_j=c'_jl^{d'_j}$ with $\gcd(c'_j,l)=1$, $1\leq j\leq r$.

We order $P_1,\ldots, P_r$ so that $n=\beta_1\geq
\beta_2\geq \ldots \geq \beta_r$.

Let $E$ be given by {\rm{(\ref{Ec1})}}. Then $E=k\Big(
\sqrt[l^{n_1}]{\*{D_1}},\ldots,\sqrt[l^{n_s}]{\*{D_s}}\Big)$.
The maximal cyclotomic extension $M$ of $E$, unramified 
at the finite primes, is
given by $M=\ge E=k\Big(\sqrt[l^{\beta_1}]{P_1^*},\ldots,
\sqrt[l^{\beta_r}]{P_r^*}\Big)$. Let $l^{m'}=e_{\infty}(M/k)$. Then
$m'=\max\limits_{1\leq j\leq r}\big\{\beta_j-\min\{\beta_j,d'_j\}\big\}$.

Choose $i$ such that $m'=\beta_i-\min\{\beta_i,d'_i\}$
and such that for $j>i$ we have $m'>\beta_j
-\min\{\beta_j,d'_j\}$. That is, $i$ is the largest
index obtaining $l^{m'}$ as the ramification index of $\p$.

In case $m'=t'=0$ we have $M=\g E=
\prod_{j=1}^r\Ku{l^{\beta_j}}{P_j^*}$ and $\g K=MK$.

In case $m'>t'\geq 0$ or $m'=t'>0$, we have $\min\{
\beta_i,d'_i\}=d'_i$ and $m'=\beta_i-d'_i$. Let $a,b\in{\ma Z}$ be such that
$a\deg P_i+b l^{n+d'_i}=l^{d'_i}=\gcd(l^{n+d'_i},\deg P_i)$. Set $z_j=-a
\frac{\deg P_j}{l^{d'_i}}=-ac'_jl^{d'_j-d'_i}\in {\ma Z}$ for $1\leq j\leq i-1$.
For $j>i$, consider $y_j\in{\ma Z}$ with $y_j\equiv -(c'_i)^{-1}c'_j\bmod l^n
\equiv -ac'_j \bmod l^n$. Let
\[
E_j=
\begin{cases}
\Ku{l^{\beta_j}}{P_jP_i^{z_j}}&\text{if $j<i$},\\
\Ku{l^{d'_i+t'-u'}}{\*{P_i}}&\text{if $j=i$},\\
\Ku{l^{\beta_j}}{P_jP_i^{y_jl^{d'_j-d'_i}}}&\text{if $j>i$ and $d'_j\geq d'_i$},\\
\Ku{l^{\beta_j+d'_i-d'_j}}{P_j^{l^{d'_i-d'_j}}{P_i^{y_j}}}&\text{if $j>i$ and $d'_i> d'_j$}.
\end{cases}
\]

Then $\g K=E_1\cdots E_{i-1}E_iE_{i+1}\cdots E_rK$,
where $l^{u'}=\frac{[\F(\sqrt[l^{n_1}]{\varepsilon_1},\ldots, 
\sqrt[l^{n_s}]{\varepsilon_s}):\F]}{\deg_K (\p)}$
and $\varepsilon_j =(-1)^{\deg D_j}\gamma_j$, $1\leq j\leq s$.
\end{theorem}

\begin{proof}
See \cite[Theorems 3.4 and 3.6]{RzeVil2021}.
\end{proof}

\subsection{Case $l\neq p$, $K/k$ non-Kummer}\label{S3.3}

In this case, we have that the extension $K/k$ is a finite abelian $l$-extension
of exponent $l^n$ such that $l^n\nmid q-1$. This case is treated in the
following section.

\section{The non-Kummer case}

Now we consider a finite abelian non-Kummer $l$-extension $K/k$ of 
exponent $l^n$. Therefore $l^n\nmid q-1$ and $\z n\notin \F$.
The non-explicit description of $\g K$ is given in
\cite[Theorem 2.2]{BaMoReRzVi2018}. Now, the description of
the subfields of a
cyclotomic function field $\cicl N{}$ is not explicit except in very few
cases. That is, if $F=k(\delta)$, it is hard to describe $\delta$ in terms
of roots of polynomial equations.
Our objective is to give explicitly the field $\gv K$, where
$w:=[\F(\z n):\F]|l^{n-1}(l-1)$. We have that 
$\v K/\v k$ and $\gv K/\v k$ are Kummer
extensions and therefore we may use Theorem \ref{T3.2} to give
$\gv K$ explicitly.

First, we recall the following non-explicit result.

\begin{proposition}\label{P3.3}
Let $F/k$ be a cyclic non-Kummer extension of prime degree $l$. 
Let $\xi\in {\mc O}_F$, the
integral closure of $R_T$ in $F$, such that $F=k(\xi)$ and
\[
\chi=\sum_{i=0}^{l-1}\zeta_l^i \varphi^i(\xi)\neq 0,
\]
where $\Gal(F/k)=\langle \varphi\rangle$. Then $\mu=\chi^l\in {\ma F}_{q^w}[T]$
and $F_w=k_w(\sqrt[l]{\mu})=k_w(\chi)$, with $w=[\F(\z n):\F]$.

\end{proposition}

\begin{proof}
See \cite[Proposition 4.1]{Wit2007}.
\end{proof}

Our first goal is to give in Theorem \ref{T4.2} an
explicit generalization of Proposition \ref{P3.3}.

As a first step, we consider the case of only one finite prime ramified.
Let $P\in R_T^+$ and consider a cyclic extension $K/k$ of degree $l^n$ with
$l^n\nmid q-1$ and such that $P$ is the
only finite prime of $k$ ramified in $K$ and it is fully ramified.
Let $w=[\F(\z n):\F]|l^{n-1}(l-1)$.
Let $d_P=\deg_k P$. Then $l^n|q^{d_P}-1$. We have that $w=\ord (q\bmod l^n)$.
Hence $w|d_P$. In the extension of constants $k_w/k$ we have that
$\con_{k/k_w}P={\mc P}_1\cdots {\mc P}_h$ where $h=\gcd(d_P,w)=w$.
That is, $P$ decomposes fully in $k_w$, and $\deg_{k_w}{\mc P}_i=d_P/w$,
$1\leq i\leq w$ (see \cite[Theorem 6.2.1]{Vil2006}).

Then $K_w/k_w$ is a Kummer extension of degree $l^n$ and the finite
primes ramified are precisely ${\mc P}_i$, $1\leq i\leq w$.
All of them are fully ramified.
Therefore there exist $\alpha_i$ such that $1\leq \alpha_i\leq l^n-1$ and
$\gcd (l,\alpha_i)=1$, $1\leq i\leq w$, with 
\[
K_w=k_w\Big(\sqrt[l^n]{\gamma\polyn w}\Big)=
k_w\big(\sqrt[l^n]{\gamma {\mc D}}\big)
=k_w(\delta)
\]
for some $\gamma\in\*{{\ma F}_{q^w}}$, ${\mc D}=\polyn w\in
{\ma F}_{q^w}[T]$ and $\delta=\sqrt[l^n]{\gamma D}$.
\[
\xymatrix{
K\ar@{-}[rr]^{\langle\tau\rangle}_w\ar@{-}[d]_{\langle\sigma\rangle}^{l^n}
&&K_w\ar@{-}[d]_{l^n}^{\langle\sigma\rangle}\\
k\ar@{-}[rr]_{\langle\tau\rangle}^w&&k_w
}
\]

Let $\Gal(K_w/K)\cong \Gal(k_w/k)=\langle \tau\rangle\cong C_w$,
and $\Gal(K_w/k_w)\cong\Gal(K/k)=\langle
\sigma\rangle\cong C_{l^n}$. Then $\ord(\tau)=w$
and we may assume that $\sigma(\delta)
=\z n \delta$. Let $\tau(\z n)=\z n^{\eta}$ where ${\eta}$ is
relatively prime to $l$ and $\ord({\eta}\bmod l^n)=w$. Since $\Gal(K_w/k)$ is an
abelian extension, $\Gal(K_w,k)=\langle\sigma,\tau\rangle$ and $\sigma\tau=
\tau\sigma$. Therefore
\[
\tau(\sigma(\delta))=\tau(\z n \delta)=\z n^{\eta} \epsilon=
\sigma(\tau(\delta))=\sigma(\epsilon),
\]
where $\epsilon:=\tau(\delta)$. Therefore we have that $\sigma(\epsilon)
=\z n^{\eta} \epsilon$. It follows that $\sigma(\delta^{-{\eta}}\epsilon)=
(\z n \delta)^{-{\eta}}\z n^{\eta} \epsilon=\delta^{-{\eta}}\epsilon$. Thus
$\delta^{-{\eta}}\epsilon\in k_w$ and $\epsilon=\lambda\delta^{\eta}$ for
some $\lambda\in k_w$.

On the other hand, since $\Gal(k_w/k)=\langle\tau\rangle$, we have
that $\tau$ acts transitively on the set $\{{\mc P}_1,\ldots, {\mc P}_w\}$.
Hence the only finite prime divisors dividing
$\epsilon^{l^n}$ are ${\mc P}_1,\ldots,{\mc P}_w$.
Without loss of generality, we may assume that
$\langle\tau\rangle$ acts as $(1,\ldots,w)$ on the set
$\{{\mc P}_1,\ldots,{\mc P}_w\}$.
That is, $\tau({\mc P}_i)={\mc P}_{i+1}$ for $i=1,2.\ldots, w-1$ and 
$\tau({\mc P}_w)={\mc P}_1$.  Thus
\begin{align}
\epsilon^{l^n}&=\tau(\delta^{l^n})=\tau(\gamma{\mc D})=\tau(\gamma\polyn w)
=\gamma^{\tau}\polynn w\nonumber\\
&=\lambda^{l^n}(\delta^{l^n})^{\eta}=\lambda^{l^n}(\gamma{\mc D})^{\eta}
=\lambda^{l^n}(\gamma \polyn w)^{\eta}.\label{Ec2}
\end{align}

It follows that, if for some finite prime divisor ${\mc P}$ we have
$v_{\mc P}(\lambda)\neq 0$, then ${\mc P}\in\{{\mc P}_1,\ldots, {\mc P}_w\}$.
Set $\xi_i:=v_{{\mc P}_i}(\lambda)$, $1\leq i\leq w$. Then, from (\ref{Ec2}) we
have
\begin{eqnarray}
\alpha_{i-1}&=&{\eta}\alpha_i+l^n\xi_i, \qquad 2\leq i\leq w,\nonumber\\
\alpha_w&=&{\eta}\alpha_1+l^n\xi_1.\label{Ec3}
\end{eqnarray}

From (\ref{Ec3}) we obtain
\begin{eqnarray*}
\alpha_{w-1}&\equiv& {\eta}\alpha_w\bmod l^n,\\
\alpha_{w-2}&\equiv& {\eta}\alpha_{w-1}\bmod l^n\equiv {\eta}^2
\alpha_w\bmod l^n,\\
\vdots&\vdots&\qquad \vdots\qquad \vdots\qquad\qquad \vdots\qquad \vdots\\
\alpha_2&\equiv&{\eta} \alpha_3\bmod l^n\equiv \ldots \equiv 
{\eta}^{w-2}\alpha_w\bmod l^n,\\
\alpha_1&\equiv&{\eta} \alpha_2\bmod l^n\equiv \ldots \equiv 
{\eta}^{w-1}\alpha_w\bmod l^n.
\end{eqnarray*}
Since $l\nmid \alpha_w$, we obtain that
\begin{align}\label{Ec4}
K_w&=k_w(\sqrt[l^n]{\gamma {\mc D}})=k_w(\delta)=
k_w\Big(\sqrt[l^n]{\gamma\big(\polynnn w\big)^{\alpha_w}}\Big)\nonumber\\
&=k_w\Big(\sqrt[l^n]{\gamma\polynnn w}\Big).
\end{align}

The extension given in (\ref{Ec4}) is determined by the class of $\gamma\in\*{
{\ma F}_{q^w}}$ modulo $({\ma F}_{q^w}^*)^{l^n}$. In particular, $K$ is cyclotomic
if and only if $(-1)^{\deg_{k_w}({\mc D})}\gamma \in ({\ma F}_{q^w}^*)^{l^n}$.

Let us obtain the ramification of the infinite primes in $K/k$. Note that since $K/k$ is an abelian
extension, if $e_{\infty}(K/k)$ denotes the ramification index of $P_{\infty}$ in
$K$, then we have $e_{\infty}(K/k)|q^{\deg P_{\infty}}-1=q-1$. In particular $P_{\infty}$ 
is not fully ramified. On the other hand, since $k_w/k$ is unramified, we obtain from
(\ref{Ec3.6}) that
\begin{gather}\label{Ec5}
e_{\infty}(K/k)=e_{\infty}(K_w/k_w)=\frac{l^n}{\gcd({\deg_{k_w}({\mc D}),l^n)}}.
\end{gather}

Now, $\deg_{k_w}({\mc D})=\sum_{i=1}^w {\eta}^{w-i}
\deg_{k_w}({\mc P}_i)=\sum_{i=1}^w
{\eta}^{w-i}d_P/w=\frac {d_P}w \frac{{\eta}^w-1}{{\eta}-1}$ 
(recall that $d_P=\deg_k(P)$), so that
\begin{gather}\label{Ec6}
e_{\infty}(K/k)=\frac{l^n}{\gcd\big(\frac {d_P}w 
\frac{{\eta}^w-1}{{\eta}-1},l^n\big)}.
\end{gather}

\subsection{General non-Kummer abelian $l$-extensions}

Now we may consider the general case. Let $K/k$ be a finite
abelian $l$-extension, where $l$ is a prime number other than the
characteristic of $k$. Let $\Gal(K/k)\cong C_{l^{n_1}}\times C_{l^{n_2}}\times
\cdots \times C_{l^{n_s}}$ with $n=n_1\geq n_2\geq\ldots\geq n_s$. Then
$\Gal(K/k)$ is of exponent $l^n$. We assume that $K/k$
is a non-Kummer extension.
However, what we will obtain, could be applied to Kummer extensions,
see Remark \ref{R4.3}.

We assume that $l^n\nmid q-1$. Let $w=[k(\z n):k]>1$. 
We have that $w=\ord(q\bmod l^n)$.
The Kummer case
is when $w=1$. We are assuming that $w\geq 2$. 

Since $K/k$ is abelian and $\deg_k(P_{\infty})=1$, $P_{\infty}$ is not fully
ramified in $K/k$.

Let $K=K_1\cdots K_{\varepsilon}\cdots K_s$ where $K_{\varepsilon}/k$ is a
cyclic extension of degree $l^{n_{\varepsilon}}$, $1\leq \varepsilon\leq s$.
Let $K_{\varepsilon,w}=K_{\varepsilon}k_w=k_w\big(\sqrt[l^{n_{\epsilon}}]{
\gamma_{\varepsilon}{\mc D}_{\varepsilon}}\big)$, with $\gamma_{\varepsilon}
\in {\ma F}_{q^w}^*$ and ${\mc D}_{\varepsilon}\in {\ma F}_{q^w}[T]$,
$1\leq \varepsilon\leq s$.
Let $P_1,\ldots,P_r\in R_T^+$ be the finite primes ramified in $K/k$ and
$\deg_k P_j=c'_jl^{d'_j}$ with $l\nmid c'_j$. 

Let $\con_{k/k_w}P_j
=\polynr j$ where $s_j=\gcd(\deg_k P_j,w)$, $1\leq j\leq r$.
Then $\deg_{k_w}({\mc P}_{j,v})=
\frac{\deg_k P_j}{s_j}=\frac{c'_jl^{d'_j}}{s_j}$ for all $1\leq v\leq s_j$. Let
$e_{P_j}(K/k)=l^{\beta_j}$, $1\leq \beta_j\leq n$, $1\leq j\leq r$ and
$e_{\infty}(K/k)=l^t$, $0\leq t\leq n$.

We have that $e_{P_j}(K_{\varepsilon}/k)=e_{P_j}(K_{\varepsilon,w}/k_w)=l^{g_{j,
\varepsilon}}$, $1\leq\varepsilon\leq s$ so that $\beta_j=\max_{1\leq\varepsilon
\leq s}\{g_{j,\varepsilon}\}$.

We have $K_{\varepsilon,w}=k_w(\sqrt[l^{n_{\varepsilon}}]{\gamma_{\varepsilon}
{\mc D}_{\varepsilon}})
=k_w\big(\sqrt[l^{n_{\varepsilon}}]{\gamma_{\varepsilon}{\mc T}_{1,\varepsilon}
{\mc T}_{2,\varepsilon}\cdots {\mc T}_{r,\varepsilon}}\big)$,
where ${\mc D}_{\varepsilon}\in{\ma F}_{q^w}[T]$ and
\begin{gather*}
{\mc T}_{j,\varepsilon}=\pol j{\varepsilon},\\
0\leq \alpha_{j,\nu,\varepsilon}\leq l^n-1,\quad
1\leq\varepsilon\leq s, \quad 1\leq j\leq r, \quad 1\leq \nu\leq s_j.
\end{gather*}

Then $g_{j,\varepsilon}=0$ if $\alpha_{j,\nu,\varepsilon}=0$
for some $1\leq \nu\leq s_j$,
and $g_{j,\varepsilon}=n_{\varepsilon}-v_l(\alpha_{j,\nu,\varepsilon})$ 
when $\alpha_{j,\nu,\varepsilon}\neq 0$ for all
$1\leq \nu\leq s_j$.

Set $\langle\sigma_{\varepsilon}\rangle\cong \Gal(K_{\varepsilon}/k)\cong
\Gal(K_{\varepsilon,w}/k_w)\cong C_{l^{n_{\varepsilon}}}$.
We have $\langle\tau\rangle\cong \Gal(K_{\varepsilon,w}/
K_{\varepsilon})\cong \Gal(k_w/k)\cong C_w$.
\[
\xymatrix{
K_{\varepsilon}\ar@{-}[rr]^{\langle\tau\rangle}_w\ar@{-}[d]_{\langle
\sigma_{\varepsilon}\rangle}^{l^{n_{\varepsilon}}}
&&K_{\varepsilon,w}\ar@{-}[d]_{l^{n_{\varepsilon}}}^{\langle\sigma_{\varepsilon}\rangle}\\
k\ar@{-}[rr]_{\langle\tau\rangle}^w&&k_w
}
\]

Since $\tau(\z n)=\z n^{\eta}$ and $\ord(\eta\bmod l^n)=w$,
we have that $\tau$ acts transitively on the set $\{{\mc P}_{j,1},
{\mc P}_{j,2},\ldots,{\mc P}_{j,s_j}\}$ and we may assume that $\tau({\mc P}_{j,\nu})=
{\mc P}_{j,\nu+1}$, $1\leq \nu\leq s_j-1$ and $\tau({\mc P}_{j,s_j})={\mc P}_{j,1}$.

Let $\delta_{\varepsilon}=\sqrt[l^{n_{\varepsilon}}]
{\gamma_{\varepsilon}{\mc D}_{\varepsilon}}$,
$\delta_{\varepsilon}^{l^{n_{\varepsilon}}}=\gamma_{\varepsilon}
{\mc D}_{\varepsilon}$.
Let $\epsilon_{\varepsilon}=\tau(\delta_{\varepsilon})$. From (\ref{Ec2}),
(\ref{Ec3}) and (\ref{Ec4}) we obtain
\begin{gather*}
\sigma(\epsilon_{\varepsilon})=\z n^{\eta}\epsilon_{\varepsilon},\quad
\epsilon_{\varepsilon}=\lambda_{\varepsilon}\delta_{\varepsilon}^{\eta},
\intertext{and}
{\mc T}_{j,\varepsilon}=\big(\poly j\big)^{\alpha_{j,\varepsilon}},
\end{gather*}
where $\alpha_{j,\varepsilon}:=\alpha_{j,s_j,\varepsilon}$, $1\leq j\leq r$ and
$1\leq \varepsilon\leq s$, and for some $\lambda_{\varepsilon}\in k_w$.
Hence 
\begin{gather}\label{Ec8}
\deg_{k_w}({\mc D}_{\varepsilon})=\sum_{j=1}^r\deg_{k_w}{\mc T}_{j,\varepsilon}
=\sum_{j=1}^r\sum_{\nu=1}^{s_j} \eta^{s_j-\nu}
\alpha_{j,\varepsilon}\frac{c'_jl^{d'_j}}{s_j}=\sum_{j=1}^r\alpha_{j,\varepsilon}
\frac{c'_jl^{d'_j}}{s_j}\frac{\eta^{s_j}-1}{\eta-1}.
\end{gather}

Therefore 
\[
e_{\infty}(K_{\varepsilon}/k)=e_{\infty}(K_{\varepsilon,w}/k_w)=
\frac{l^{n_{\varepsilon}}}{\gcd(\deg_{k_w}({\mc D}_{\varepsilon}),
l^{n_{\varepsilon}})}=
l^{n_{\varepsilon}-\min\{n_{\varepsilon},v_l(\deg_{k_w}(
{\mc D}_{\varepsilon}))\}}
\]
where $\deg_{k_w}({\mc D}_{\varepsilon})$ is given by (\ref{Ec8}).
Hence, 
\begin{gather}
e_{\infty}(K/k)=e_{\infty}(K_w/k_w)=l^t\nonumber
\intertext{with}
t=\max_{1\leq \varepsilon\leq s}\{n_{\varepsilon}-\min\{n_{\varepsilon},v_l(\deg_{k_w}(
{\mc D}_{\varepsilon}))\}\}.\label{Ec9}
\end{gather}

Now, since $e_{P_j}(K/k)=e_{P_j}(E/k)=l^{\beta_j}$, where
$E$ is given by (\ref{Ec1}), we have that
$\ge E=\prod_{j=1}^rF_j$ with $k\subseteq F_j\subseteq \cicl {P_j}{}$
and $[F_j:k]=l^{\beta_j}$.

From (\ref{Ec4}), we obtain that
$F_{j,w}=k_w\Big(\sqrt[l^{\beta_j}]{\*{\big(\poly
j\big)}}\Big)=
k_w\big(\sqrt[l^{\beta_j}]{{\mc Q}_j^*}\big)$, where
\begin{gather}\label{Ec10}
{\mc Q}_j:=\poly j,\quad 1\leq j\leq r.
\end{gather}

We have 
\begin{align*}
e_{\infty}(F_j/k)=e_{\infty}(F_{j,w}/k_w)&=\frac{l^{\beta_j}}
{\gcd(\deg_{k_w}({\mc Q}_j),l^{\beta_j})}
=\frac{l^{\beta_j}}{\gcd(
l^{v_l(\deg_{k_w}({\mc Q}_j))},l^{\beta_j})}\\
&=l^{\beta_j-\min\{\beta_j,
v_l(\deg_{k_w}({\mc Q}_j))\}},
\end{align*}
where $\deg_{k_w}({\mc Q}_j)=\frac{c'_jl^{d'_j}}{s_j}\frac{\eta^{s_j}-1}{\eta-1}$
(see (\ref{Ec8})). Therefore
\begin{align*}
l^m:&=e_{\infty}(\ge K/k)=e_{\infty}(\gev K/k_w)=e_{\infty}(\ge E/k)\\
&=e_{\infty}(\gev E/k)=\lcm_{1\leq j\leq r}e_{\infty}(F_j/k).
\end{align*}
Hence
\begin{gather}\label{Ec11}
m=\max_{1\leq j\leq r}\{\beta_j-\min\{\beta_j,v_l(\deg_{k_w}({\mc Q}_j))\}\}.
\end{gather}

Now, $\g K$ is the extension $K\subseteq \g K\subseteq \ge K$ such that
$e_{\infty}(\g K/K)=1$ and $[\ge K:\g K]=l^{m-t}$. Thus, by the Galois 
correspondence, $\gv K$ is the extension $K_w\subseteq \gv K\subseteq 
\gev K$ such that
$e_{\infty}(\gv K/K_w)=1$ and $[\gev K:\gv K]=l^{m-t}$. 

Our main result is the explicit description of $\gv K$.

\begin{theorem}\label{T4.2}
Let $K/k$ be a finite non-Kummer $l$--extension 
of $k$ with Galois group  $\Gal(K/k)\cong C_{l^{n_1}}\times 
\cdots \times C_{l^{n_s}}$ where $n=n_1\geq n_2
\geq \cdots \geq n_s$ and $l^n\nmid q-1$. Let
$K=K_1\cdots K_s$ be such that $\Gal(K_{\varepsilon}/k)
\cong C_{l^{n_{\varepsilon}}}$, $1\leq \varepsilon\leq s$.

Let $P_1,\ldots,P_r$ be the finite primes in $k$ ramified in
$K$ with $P_1,\ldots,P_r\in R_T^+$ distinct. Let 
$\deg P_j=c'_jl^{d'_j}$ with $\gcd(c'_j,l)=1$, $1\leq j\leq r$. Let
\begin{gather*}
e_{P_j}(K/k)=l^{\beta_j},\quad
1\leq \beta_j\leq n, \quad 1\leq j\leq r,\quad\text{and}\quad
e_{\infty}(K/k)=l^{t}, \quad 0\leq t\leq n
\end{gather*}
given by {\rm{(\ref{Ec9})}}.
We order $P_1,\ldots, P_r$ so that $n=\beta_1\geq
\beta_2\geq \ldots \geq \beta_r$.

Let $E$ be given by {\rm{(\ref{Ec1})}}. 
The maximal cyclotomic extension $M$ of $E$, unramified 
at the finite primes, is
given by $M=\ge E= 
\prod_{j=1}^r F_j$ where $F_j$ is the only field satisfying
$k\subseteq F_j\subseteq \cicl {P_j}{}$ and $[F_j:k]=l^{\beta_j}$.

Let $w=[\F(\z n):\F]>1$, $\langle\tau\rangle=\Gal(k_w/k)$ with
$\tau(\z n)=\z n^{\eta}$, where $\ord(\eta\bmod l^n)=w$. Let
$\con_{k/k_w}P_j =\polynr j$, $1\leq j\leq r$ and
\[
{\mc Q}_j:= \poly j\in {\ma F}_{q^w}[T].
\]

Set $\deg_{k_w}({\mc Q}_j)=c_jl^{d_j}$ with $l\nmid c_j$.
Then we have $F_{j,w}=k_w\big(\sqrt[l^{\beta_j}]{{\mc Q}_j^*}\big)$
and $K_{\varepsilon,w}=k_w\big(\sqrt[l^{n_{\varepsilon}}]{
\gamma_{\varepsilon} {\mc D}_{\varepsilon}}\big)$ for some
$\gamma_{\varepsilon}\in {\ma F}_{q^w}^*$ and 
${\mc D}_{\varepsilon}=
\prod_{j=1}^r {\mc T}_{j,\varepsilon}$ where ${\mc T}_{j,\varepsilon}$
is given by ${\mc T}_{j,\varepsilon}=
\big(\poly j\big)^{\alpha_{j,\varepsilon}}$ where
$0\leq \alpha_{j,\varepsilon}
\leq l^{n_{\varepsilon}}-1$.

Let $l^m=e_{\infty}(\ge K/k)=e_{\infty}(\gv K/k_w)$. Then
$m=\max\limits_{1\leq j\leq r}\big\{\beta_j-\min\{\beta_j,d_j\}\big\}$,
(see {\rm{(\ref{Ec11})}}).

Choose $i$ such that $m=\beta_i-\min\{\beta_i,d_i\}$
and such that for $j>i$ we have $m>\beta_j
-\min\{\beta_j,d_j\}$. That is, $i$ is the largest
index obtaining $l^{m}$ as the ramification index of 
$\con_{k/k_w}\p={\mc P}_{\infty}$,
the infinite prime of ${\ma F}_{q^w}[T]$.

In case $m=t=0$ we have $M=\g E=\prod_{j=1}^r F_j$, $\gv E=
\prod_{j=1}^rk_w\big(\sqrt[l^{\beta_j}]{{\mc Q}_j^*}\big)$ 
and $\gv K=\gv E K_w$.

In case $m>t\geq 0$ or $m=t>0$, we have $\min\{
\beta_i,d_i\}=d_i$ and $m=\beta_i-d_i$. Let $a,b\in{\ma Z}$ be such that
$a\deg_{k_w}({\mc Q}_i)+b l^{n+d_i}=l^{d_i}=\gcd(l^{n+d_i},
\deg_{k_w}( {\mc Q}_i))$. Set $z_j=-a
\frac{\deg_{k_w}({\mc Q}_j)}{l^{d_i}}=-ac_jl^{d_j-d_i}\in 
{\ma Z}$ for $1\leq j\leq i-1$.
For $j>i$, consider $y_j\in{\ma Z}$ with $y_j\equiv -{c}_i^{-1}c_j\bmod l^n
\equiv -ac_j \bmod l^n$. Let
\[
L_j=
\begin{cases}
\Kuu{l^{\beta_j}}{{\mc Q}_j{\mc Q}_i^{z_j}}&\text{if $j<i$},\\
\Kuu{l^{d_i+t-u-v}}{\*{{\mc Q}_i}}&\text{if $j=i$},\\
\Kuu{l^{\beta_j}}{{\mc Q}_j{\mc Q}_i^{y_jl^{d_j-d_i}}}&\text{if $j>i$ and $d_j\geq d_i$},\\
\Kuu{l^{\beta_j+d_i-d_j}}{{\mc Q}_j^{l^{d_i-d_j}}{{\mc Q}_i^{y_j}}}&
\text{if $j>i$ and $d_i> d_j$},
\end{cases}
\]
where $l^{u}=\frac{[{\ma F}_{q^w}(\sqrt[l^{n_1}]{\mu_1},\ldots, 
\sqrt[l^{n_s}]{\mu_s}):{\ma F}_{q^w}]}{\deg_{K_w}({\mc P}_{\infty})}$, 
$\mu_{\varepsilon} :=(-1)^{\deg_{k_w}({\mc D}_{\varepsilon})}
\gamma_{\varepsilon}$, $1\leq {\varepsilon}\leq s$,
$\deg_{k_w}({\mc P}_{\infty})=\frac{\deg_K(P_{\infty})}{\gcd(\deg_K(
P_{\infty}),w)}$
and $l^v=\frac{\gcd([EK:K],w)}{\gcd(\deg_K(P_{\infty}),w)}$.
Equivalently, 
\[
l^{u+v}=\frac{[{\ma F}_{q^w}(\sqrt[l^{n_1}]{\mu_1},\ldots, 
\sqrt[l^{n_s}]{\mu_s}):{\ma F}_{q^w}]\cdot \gcd([EK:K],w)}{\deg_K(P_{\infty})}.
\]

Then $\gv K=L_1\cdots L_{i-1}L_iL_{i+1}\cdots L_rK$.
\end{theorem}

\begin{proof}
Let $H$ and $H'$ be the decomposition groups of $EK/K$ 
and of $E_wK_w/K_w$, respectively.
Set $|H|=l^{\xi}$ and $|H'|=l^u$. We have $u\leq \xi$. Theorem \ref{T3.2} gives
the value of $u$. That is $l^u=\frac{[{\ma F}_{q^w}(\sqrt[l^{n_1}]{\mu_1},\ldots, 
\sqrt[l^{n_s}]{\mu_s}):{\ma F}_{q^w}]}{\deg_{K_w}({\mc P}_{\infty})}$.
Set $\xi=u+v$, that is, $|H|=|H'|l^v$, for some $v\geq 0$.

From \cite[proof of Theorem 2.2]{BaMoReRzVi2018} 
we have that $EK/K$ is an extension of
constants and since we are assuming that the field of constants of $K$ is 
$\F$, we have that the field of constants of $EK$ is ${\ma F}_{q^{\chi}}$ where
$\chi=[EK:K]=[E^HK:K]|H|$. By the same reason, $EK/E$ is an extension of
constants and $[EK:E]=f_{\infty}(EK/k)=\deg_{EK}(P_{\infty}) =f_{\infty}(EK/K)
f_{\infty}(K/k)=|H|\deg_K(P_{\infty})$.

Therefore, $[E^H:K]=\deg_K(P_{\infty})$, $[EK:K]=[EK:E]$ and ${\ma F}_{q^{
|H|\deg_K(P_{\infty})}}$ is the field of constants of $EK$.

Since $E_wK_w=(EK)_w=EK{\ma F}_{q^w}$, the field of constants of $E_wK_w$
is 
\begin{gather}\label{Ec12}
{\ma F}_{q^{\lcm[|H|\deg_K(P_{\infty}),w]}}=
{\ma F}_{q^{\lcm[[EK:K],w]}}.
\end{gather}

 On the other hand, $E_wK_w/K_w$ is
an extension of constants of degree $[E_wK_w:K_w]=|H'|
\deg_{k_w}({\mc P}_{\infty})$.
Since the field of constants of $K_w$ is ${\ma F}_{q^w}$, we have that the field
of constants of $E_wK_w$ is
\begin{gather}\label{Ec13}
{\ma F}_{q^{|H'|\deg_{k_w}({\mc P}_{\infty}) w}}=
{\ma F}_{q^{|H'|\lcm[\deg_K(P_{\infty}),w]}}.
\end{gather}

From (\ref{Ec12}) and (\ref{Ec13}), we obtain
\begin{gather*}
\lcm[|H|\deg_K(P_{\infty}),w]=|H'|\lcm[\deg_K(P_{\infty}),w].
\intertext{Now}
\begin{align*}
\lcm[|H|\deg_K(P_{\infty}),w]&=\frac{|H|\deg_K(P_{\infty})w}{
\gcd(|H|\deg_K(P_{\infty}),w)}\\
&=|H|\frac{\lcm(\deg_K(P_{\infty}),w)
\gcd(\deg_K(P_{\infty}),w)}{\gcd(|H|\deg_K(P_{\infty}),w)}.
\end{align*}
\intertext{Therefore}
|H'|=|H|\frac{\gcd(\deg_K(P_{\infty}),w)}{\gcd(
|H|\deg_K(P_{\infty}),w)}=|H|\frac{\gcd(\deg_K(P_{\infty}),w)}{\gcd(
[EK:K],w)}.
\end{gather*}

Hence $l^v=\frac{\gcd(
[EK:K],w)}{\gcd(\deg_K(P_{\infty}),w)}$.

The rest of the proof follows the lines of the proof of Theorem \ref{T3.2} (see
\cite[Theorem 3.4]{RzeVil2021}).
\end{proof}

\begin{remark}\label{R4.3}{\rm{
In Theorem \ref{T4.2}, the Kummer case occurs when
we allow $w=1$ and Theorem
\ref{T3.2} in this sense can be considered
the special case $w=1$ in Theorem \ref{T4.2}.
}}
\end{remark}

\section{Genus fields of finite abelian extensions of $k$}

In \cite[Theorem  4.1]{RzeVil2021}, we obtained the following result.

\begin{theorem}\label{T5.1}
Let $J_i/k$, $i=1,2$ be two finite abelian extensions such that
$\gcd([J_1:k],[J_2:k])=1$. Then $\g{(J_1)}\g{(J_2)}=\g{(J_1
J_2)}$. $\fin$
\end{theorem}

As a consequence of the previous results, we have the explicit
description of any finite abelian $K/k$.  For a finite non-trivial $l$-group
$S$, we denote by $\exp(S)=l^n$ where $n$ is the minimum
natural number $n$ such that $S^{l^n}=\{1\}$.

\begin{theorem}[Genus field of an abelian extension]\label{T5.2}
Let $K/k$ be a finite abelian extension with Galois group $G=\Gal
(K/k)$. Let $S_1,S_2,\ldots,S_t$ be the different Sylow subgroups
of $G$ with $S_j$ the $l_j$-Sylow subgroup of $G$. Let $K=K_1\cdots
K_t$ be such that $\Gal(K_j/k)\cong S_j$, $1\leq j\leq t$. Then
\[
\g K=\prod_{j=1}^t \g{(K_j)},
\]
where $\g{(K_j)}$ is given by
\[
\begin{cases}
\text{Theorem {\rm{\ref{T3.1}}}}&\text{if $l_j=p$},\\
\text{Theorem {\rm{\ref{T3.2}}}}&\text{if $l_j\neq p$ and $\exp(S_j)|q-1$},\\
\text{Theorem {\rm{\ref{T4.2}}}}&\text{if $l_j\neq p$ and $\exp(S_j)\nmid q-1$}.
\end{cases}
\]
\end{theorem}

\begin{proof}
The result is an immediate consequence of Theorems \ref{T3.1}, \ref{T3.2}, \ref{T4.2}
and \ref{T5.1}.
\end{proof}


\begin{thebibliography}{xxxx}

\bibitem{BaMoReRzVi2018} Barreto--Casta\~neda, Jonny
Fernando; Montelongo--V\'azquez, Carlos; Reyes--Morales
Carlos Daniel; Rzedowski--Calder\'on, Martha \& 
Villa--Salvador, Gabriel Daniel, \textit{Genus fields of abelian
extensions of rational congruence function fields II}, Rocky
Mountain Journal of Mathematics, \textbf{48}, no. 7, 2099--2133 (2018).

\bibitem{BaRzVi2013} Bautista--Ancona, V\'ictor; Rzedowski--Calder\'on
Martha \& Villa--Salvador Gabriel, \textit{Genus fields of cyclic
$l$--extensions of rational function fields}, International Journal
of Number Theory, \textbf{9}, no. 5, 1249--1262 (2013).

\bibitem{Leo53} Leopoldt, Heinrich W., \textit{Zur Geschlechtertheorie in abelschen
Zahlk\"orpern}, Math. Nachr. \textbf{9}, 351--362 (1953).

\bibitem{MaRzVi2013} Maldonado--Ram{\'\i}rez, Myriam; Rzedowski--Calder\'on,
Martha \& Villa--Salvador, Gabriel, 
\textit{Genus Fields of Abelian Extensions of Congruence
Rational Function Fields}, Finite Fields Appl. \textbf{20}, 40--54 (2013).

\bibitem{MaRzVi2017} Maldonado--Ram{\'\i}rez, Myriam; Rzedowski--Calder\'on,
Martha \& Villa--Salvador, Gabriel, 
\textit{Genus Fields of Congruence
Function Fields}, Finite Fields Appl. \textbf{44}, 56--75 (2017).

\bibitem{MoReVi2019} Reyes--Morales,
Carlos \& Villa--Salvador, Gabriel \textit{Genus fields of Kummer
$\ell^n$--cyclic extensions}, International Journal of Mathematics
\textbf{32}, no. 9, (2021), paper no. 2150062, 21 p.

\bibitem{RzeVil2017} Rzedowski--Calder\'on, Martha \&
Villa--Salvador, Gabriel, \textit{Campos ciclot\'omicos
num\'ericos y de funciones (segunda versi\'on)},
{\tt https://arxiv.org/abs/1407.3238}.

\bibitem{RzeVil2021} Rzedowski--Calder\'on, Martha \&
Villa--Salvador, Gabriel, \textit{Genus fields of Kummer extensions of
rational function fields}, Finite Fields Appl. \textbf{77}, (2022), paper no. 101943,
19 p.

\bibitem{Vil2006} Villa--Salvador, Gabriel, \textit{Topics in the theory of
algebraic function fields}, Mathematics: Theory \& Applications. Birkh\"auser Boston,
Inc., Boston, MA, 2006.

\bibitem{Wit2007} Wittmann, Christian, \textit{$l$-Class groups 
of cyclic function fields of degree $l$}, Finite Fields Appl. \textbf{13}, 
327--347 (2007).

\end{thebibliography}
\end{document}